\documentclass[12pt,oneside,english]{amsart}
\usepackage[T1]{fontenc}
\usepackage[latin9]{inputenc}
\usepackage{geometry}
\usepackage{float}
\usepackage{url}
\usepackage{amstext}
\usepackage{amsthm}
\usepackage{amssymb}
\usepackage{comment}

\makeatletter
\numberwithin{equation}{section}
\numberwithin{figure}{section}
\theoremstyle{plain}
\newtheorem{thm}{\protect\theoremname}[section]
\theoremstyle{definition}

\theoremstyle{plain}
\newtheorem{rem}[thm]{\protect\remarkname}
\theoremstyle{plain}
\newtheorem{prop}[thm]{\protect\propositionname}
\theoremstyle{plain}
\newtheorem{lem}[thm]{\protect\lemmaname}

\usepackage[svgnames]{xcolor}
\usepackage[bookmarksnumbered=true]{hyperref} 
\hypersetup{
	colorlinks = true,
	linkcolor = Blue,
	anchorcolor = blue,
	citecolor = Green,
	filecolor = blue,
	urlcolor = FireBrick
}

\makeatother

\usepackage{babel}
\providecommand{\definitionname}{Definition}
\providecommand{\lemmaname}{Lemma}
\providecommand{\propositionname}{Proposition}
\providecommand{\remarkname}{Remark}
\providecommand{\theoremname}{Theorem}

\begin{document}
\title[Optimality of increasing stability for an IBVP]{Optimality of increasing stability for an inverse boundary value problem}
\author[Kow]{Pu-Zhao Kow}
\address{Department of Mathematics, National Taiwan University, Taipei 106,
Taiwan. }
\email{d07221005@ntu.edu.tw}
\author[Uhlmann]{Gunther Uhlmann}
\address{Department of Mathematics, University of Washington, Box 354350, Seattle, WA 98195-4350, USA, and Institute for Advanced Study of
the Hong Kong University of Science and Technology, Hong Kong.}
\email{gunther@math.washington.edu}
\author[Wang]{Jenn-Nan Wang}
\address{Institute of Applied Mathematical Sciences, National Taiwan
University, Taipei 106, Taiwan. }
\email{jnwang@math.ntu.edu.tw}
\begin{abstract}
In this work we study the optimality of increasing stability of the inverse boundary value problem (IBVP) for the Schr\"{o}dinger equation. The rigorous justification of  increasing stability for the IBVP for the Schr\"{o}dinger equation were established by Isakov \cite{Isa11} and by Isakov, Nagayasu, Uhlmann, Wang of the paper \cite{INUW14}. In \cite{Isa11}, \cite{INUW14}, the authors showed that the stability of this IBVP increases as the frequency increases in the sense that the stability estimate changes from a logarithmic type to a H\"{o}lder type. In this work, we prove that the instability changes from an exponential type to a H\"older type when the frequency increases. This result verifies that results in \cite{Isa11}, \cite{INUW14} are optimal. 
\end{abstract}

\maketitle

\section{Introduction}

In this paper, we study the instability phenomenon of the inverse boundary value
problem (IBVP) for the Schr\"{o}dinger equation with a frequency. This IBVP is notoriously ill-posed. Ignoring the effect of the frequency, a logarithmic type estimate were first derived in \cite{Ale88} and this logarithmic estimate was shown be optimal in the form of exponential instability in \cite{Man01}. Such exponential instability was also established for different inverse problems in \cite{DR03, RS18, ZZ19}. However,
by taking account of the frequency, it was observed numerically \cite{CHP03} or \cite{KW20} that the stability of some inverse problems will improve as the frequency increases. The increasing stability phenomena were rigorously proved in other situations \cite{DI07,DI10,HI04,ILX20,INUW14,Isa07,Isa11,KU19,LLU19,NUW13}, not only for inverse problems, but also for the unique continuation property. 

In this work, we will study the counterpart of the increasing stability by investigating how the exponential instability is affected by the frequency. Here we consider the Schr\"{o}dinger equation with a potential and a frequency
\begin{equation}
\bigg[\Delta+q(x)+\kappa^{2}\bigg]u=0\quad\text{in}\;\;\Omega\subset\mathbb{R}^{d}\label{eq:sch-general}
\end{equation}
with $d\ge 2$. The Cauchy data corresponding to the Schr\"{o}dinger equation \eqref{eq:sch-general}
is defined by 
\[
\mathcal{C}_{q}:=\begin{Bmatrix}\begin{array}{l|l}
{\displaystyle \bigg(u|_{\partial\Omega},\frac{\partial}{\partial\nu}u\bigg|_{\partial\Omega}\bigg)} & u\text{ is a solution to \eqref{eq:sch-general}}\end{array}\end{Bmatrix}.
\]
Let ${\rm dist}\,(\mathcal{C}_{q_{1}},\mathcal{C}_{q_{2}})$
be the Hausdorff distance between two Cauchy data with respect to $q_1, q_2$. Under some appropriate assumptions, it was shown in \cite{INUW14} that 
\begin{equation}
\|\tilde{q}\|_{H^{-s}(\mathbb{R}^{d})}\le C\bigg(\kappa+\log\frac{1}{{\rm dist}\,(\mathcal{C}_{q_{1}},\mathcal{C}_{q_{2}})}\bigg)^{-2s-d}+C\kappa^{4}{\rm dist}\,(\mathcal{C}_{q_{1}},\mathcal{C}_{q_{2}}),\label{eq:sch-stable}
\end{equation}
where $\tilde{q}$ is the zero extension of $q_{1}-q_{2}$. The estimate \eqref{eq:sch-stable} clearly indicates that the logarithmic part decreases as $\kappa$ increases and the estimate changes from a logarithmic type to a H\"older type. Isakov in \cite{Isa11}, as well as Isaev and Novikov in \cite{IN12}, proved a similar estimate in terms of the Dirichlet-to-Neumann map. 

We now briefly describe the problem considered in this work. For simplicity, we consider \eqref{eq:sch-general} in $\Omega=B_{1}$ and $\partial\Omega=\mathcal{S}^{d-1}=\partial B_{1}$. If $\kappa^{2}\notin{\rm spec}\,(-(\Delta+q))$, where ${\rm spec}(-(\Delta+q))$ denotes the Dirichlet spectra of $-(\Delta+q)$ in $\Omega$, then the Dirichlet-to-Neumann map of the Schr\"{o}dinger equation \eqref{eq:sch-general},  $\Lambda_q: u|_{\mathcal{S}^{d-1}} \mapsto\partial_{r}u|_{\mathcal{S}^{d-1}}$ is well-defined.
It is well-known that 
\[
\Lambda_q\in\mathcal{L}(H^{\frac{1}{2}}(\mathcal{S}^{d-1}),H^{-\frac{1}{2}}(\mathcal{S}^{d-1})),
\]
where $\mathcal{L}(\mathcal{X},\mathcal{Y})$ denotes the space of
bounded linear operators from $\mathcal{X}$ to $\mathcal{Y}$. Clearly, given any $s\ge\frac{1}{2}$, we have 
\[
\Lambda_q\in\mathcal{L}(H^{s}(\mathcal{S}^{d-1}),H^{-s}(\mathcal{S}^{d-1})).
\]
To simplify the notations, we simply denote 
\[
\|\bullet\|_{\mathcal{L}(H^{s}(\mathcal{S}^{d-1}),H^{-s}(\mathcal{S}^{d-1}))}=\|\bullet\|_{s\rightarrow-s}.
\]

We now state the main results of this paper. 

\begin{thm}
	\label{thm:main}Let $\alpha>0$, $R>0$, and any 
	\begin{equation}
		\kappa^{2} > 0  \label{eq:wave-number}
	\end{equation}
	and 
	\begin{equation}
		\text{perturbation}\quad\theta\in(0,\min\{(2+\log2)^{-2\alpha},R\}),\label{eq:perturbation}
	\end{equation}
	where $R$ is related to the size of the potentials. Then there exist potentials $q_{1},q_{2}\in C^\alpha(B_{1})$ such that
	\begin{equation}
		\|q_{1}-q_{2}\|_{L^{\infty}(B_{1})} \ge\theta\label{eq:discrete-q1-q2}
	\end{equation}
	and
	\begin{equation}
		\|\Lambda_{q_{1}}-\Lambda_{q_{2}}\|_{\frac{d+4}{2}\rightarrow-\frac{d+4}{2}}\le C_{R}(1+\kappa^{2})\bigg[\exp\bigg(-\frac{1+\kappa^{2}}{3}\theta^{-\frac{1}{2\alpha}}\bigg)+3\theta^{\frac{1}{2\alpha}}\bigg],\label{eq:DN-map-approx}
	\end{equation}
	where $C_{R}$ is a positive constant depends only on the parameter
	$R$.
\end{thm}
\begin{rem}
There are two reasons in using $H^{\frac{d+4}{2}}$ norm in the Dirichlet-to-Neumann map. One is to bound the $\|\cdot\|_{\frac{d+4}{2}\rightarrow-\frac{d+4}{2}}$ norm of the Dirichlet-to-Neumann map by the $X_s$ of its corresponding discretization (see \eqref{eq:net}). The other is to bound the cardinality of the $\delta$-net $Y$ 
(see {\rm Remark~\ref{fk}}).
\end{rem}

\begin{rem}\label{rem:lambda-choice}
In the proof, we actually construct $q_{\ell}$ ($\ell=1,2$) with $q_{\ell} = \tilde{q}_{\ell} + i$ for some real-valued functions $\tilde{q}_{\ell}$, where $i = \sqrt{-1}$. Therefore, $\kappa^{2}\notin\cup_{\ell=1,2}{\rm spec}\,(-(\Delta+q_{\ell}))$ and so $\Lambda_{q_{\ell}}$ are well-defined. We want to point out that the discrepancy parameter $\theta$ is independent of $\kappa$ and the estimate on the right hand side of \eqref{eq:DN-map-approx} is expressed explicitly in $\kappa$. Furthermore, $C_R$ in \eqref{eq:DN-map-approx} only depends on the size of the potentials and is independent of $\kappa$. The assumption of $\Im (q_\ell)=1$ is to achieve this purpose with the help of explicit elliptic estimates.  On the other hand, in the usual stability estimate of the inverse boundary value problem for the Schr\"odinger equation, potentials $q_\ell$, $\ell=1,2$, are a priori given such that we can assume that $\kappa^2$ is not a Dirichlet eigenvalue of $-(\Delta+q_\ell)$. In which case, we can consider the Dirichlet-to-Neumann maps \cite{Isa11}. Instead of Dirichlet-to-Neumann maps, we can also use the Cauchy data in the stability estimate \cite{INUW14}. However, in the instability estimate, potentials $q_\ell$ are determined posteriorly. Therefore, either assuming that $\kappa^2$ is not a Dirichlet eigenvalue of $-(\Delta+q_\ell)$ or considering the Cauchy data are not feasible.
 
\end{rem}

By the standard elliptic estimates, one can show that for \emph{any} $q_1,q_2\in L^\infty(B_1)$
\begin{equation}\label{well-posed}
\|\Lambda_{q_1}-\Lambda_{q_2}\|_{\frac 12\to-\frac 12}\le C(q_1,q_2,\kappa)\|q_1-q_2\|_{L^\infty(B_1)}
\end{equation}
provided $\Lambda_{q_1}$ and $\Lambda_{q_2}$ are well-defined, where the constant $C(q_1,q_2,\kappa)$ depends on $q_1,q_2,$ and $\kappa$. Estimate \eqref{well-posed} can be viewed as an well-posedness estimate of the mapping $q\to\Lambda_q$. However, it provides no information of the ill-posedness of $\Lambda_q\to q$. 

Theorem~\ref{thm:main} is valid for all frequencies $\kappa$. The estimate \eqref{eq:DN-map-approx} apparently demonstrates that its right hand side is dominated by the H\"older component $\theta^{\frac{1}{2\alpha}}$ when $\kappa$ is large. 
To further elucidate our result, it is interesting to compare Theorem~\ref{thm:main} and a similar instability estimate proved by Isaev in \cite{Isa13}. In our work, we first give an \emph{arbitrary} wave number $\kappa^{2}$ and a perturbation $\theta$ (independent of $\kappa$), and then construct suitable potentials $q_{1},q_{2}$ satisfying \eqref{eq:discrete-q1-q2} and \eqref{eq:DN-map-approx}. In \cite{Isa13}, Isaev constructed \emph{some} wave number $\kappa^{2}$ and some potentials $q \in \mathcal{C}^{d+1}(B_{1})$ such that 
\begin{equation}
\|q\|_{L^{\infty}(B_{1})} > (1 + \kappa)\delta + (1 + \kappa)^{-(d+2)}(\ln(3+\delta^{-1}))^{-\frac{d+2}{2}}, \label{eq:instability}
\end{equation}
where $\delta = \|\Lambda_{q}-\Lambda_{0}\|_{L^{\infty}(\partial B_{1})\rightarrow L^{\infty}(\partial B_{1})}$. In other words, \eqref{eq:instability} validates the optimality of the increasing stability estimate obtained in \cite{IN12} (similar to \eqref{eq:sch-stable}) for potentials near zero. We want to point out that the potentials $q_1, q_2$ constructed in Theorem ~\ref{thm:main} are not necessarily small. Unlike the $\kappa$-independent discrepancy parameter $\theta$ in \eqref{eq:discrete-q1-q2}, the lower bound of $\|q\|_{L^{\infty}(B_{1})}=\|q-0\|_{L^{\infty}(B_{1})}$ depends on $\kappa$ in view of \eqref{eq:instability}.

However, Theorem~\ref{thm:main} is not optimal in the lower frequency. It is expected that the IVBP is exponentially unstable if $\kappa$ is small. To verify this, we prove a similar estimate in the lower frequency. 
\begin{thm}
	\label{thm:second-main}Let $\alpha>0$ and 
	\[
	0<\kappa^{2}\le\frac{1}{4}\kappa_{1}
	\]
	and 
	\[
	\text{perturbation}\quad\theta\in \bigg(0,\min \bigg\{ (\frac{\kappa_1}{2}+\log2)^{-2\alpha},\frac{1}{4}\kappa_{1} \bigg\}\bigg),
	\]
	where $\kappa_1$ is the first Dirichlet eigenvalue of $-\Delta$ on $B_1$. Then there exist (real-valued) potentials $q_{1},q_{2}\in C^\alpha(B_{1})$
	such that \eqref{eq:discrete-q1-q2} holds and
	\begin{equation}
		\|\Lambda_{q_{1}}-\Lambda_{q_{2}}\|_{\frac{d+4}{2}\rightarrow-\frac{d+4}{2}}\le 16\sqrt{2}\exp\bigg(-\frac{1}{3}\theta^{-\frac{1}{2\alpha}}\bigg)+24\sqrt{2}\kappa^{2}\theta^{\frac{1}{2\alpha}}.\label{eq:DN-est-improve}
	\end{equation}
\end{thm}

\begin{rem}
	Remarked as above, we construct real-valued potentials $q_{1},q_{2}\in L^{\infty}(B_{1})$
	with 
	\[
	\|q_{l}\|_{L^{\infty}}+\kappa^2<\kappa_{1},\;\; l=1,2.
	\] 
	Therefore,
	$\kappa^{2}\notin\cup_{l=1,2}{\rm spec}\,(-(\Delta+q_{l}))$ and the Dirichlet-to-Neumann map $\Lambda_{q_l}$ is well-defined.
\end{rem}

Estimate \eqref{eq:DN-map-approx} shows that the instability changes from an exponential type to a H\"older type when $\kappa^2$ increases, conversely, \eqref{eq:DN-est-improve} shows that the instability is an exponential type when $\kappa^2$ is small as in \cite{Man01}. In particular, for any perturbation $\theta$, we have the following dichotomy:
\begin{equation*}
\left\{
\begin{aligned}
	&\mbox{exponential instability when}\;\;\kappa^{2} < \theta_{*}\quad \bigg( \mbox{i.e. when}\;\;\kappa^{2} \;\;\mbox{is small} \bigg),\\
	&\mbox{H\"older instability when}\;\;\kappa^{2} > \theta^{*}\quad \bigg( \mbox{i.e. when}\;\;\kappa^{2} \;\;\mbox{is large}\bigg),
\end{aligned}\right.
\end{equation*}
where
\begin{align*}
	\theta_{*}& = \min \bigg\{ 3\theta^{\frac{1}{2\alpha}}\log\bigg(\frac{1}{3}\theta^{-\frac{1}{2\alpha}}\bigg) , \frac{2}{3} \theta^{\frac{1}{2\alpha}} \exp \bigg( -\frac{1}{3} \theta^{-\frac{1}{2\alpha}}\bigg) \bigg\}, \\
	\theta^{*}& = \max \bigg\{ 3\theta^{\frac{1}{2\alpha}}\log\bigg(\frac{1}{3}\theta^{-\frac{1}{2\alpha}}\bigg) , \frac{2}{3} \theta^{\frac{1}{2\alpha}} \exp \bigg( -\frac{1}{3} \theta^{-\frac{1}{2\alpha}}\bigg) \bigg\}. 
\end{align*}
Such transition of instability was also proved for an inverse problem in the stationary radiative transport equation in \cite{ZZ19}. Our study is also inspired by their result. We also want to comment that the stability estimate of determining a potential in the wave equation by the knowledge of the hyperbolic Dirichlet-to-Neumann map was shown to be a H\"older type, see \cite{Sun90}. In other words, in the high frequency, the inverse boundary value for \eqref{eq:sch-general} is as stable as the inverse problem for the wave equation. 

We can actually combine Theorem~\ref{thm:main} and Theorem~\ref{thm:second-main} into a single theorem. 
\begin{thm}
Let $\alpha>0$ and $R>0$ be any given constants. There exists a positive constant $\theta_{0} = \theta_{0}(\alpha,R)$ such that the following statement holds: For each $\kappa > 0$ and $0 < \theta < \theta_{0}$, there exist potentials $q_{1},q_{2} \in \mathcal{C}^{\alpha}(B_{1})$ such that \eqref{eq:discrete-q1-q2} holds and 
\begin{equation}
\|\Lambda_{q_{1}}-\Lambda_{q_{2}}\|_{\frac{d+4}{2}\rightarrow-\frac{d+4}{2}}\le C_{R} \max\{1,\kappa^{2}\} \exp \bigg( -c_{0} \max\{1,\kappa^{2}\}\theta^{-\frac{1}{2\alpha}} \bigg) + C_R\kappa^{2}\theta^{\frac{1}{2\alpha}}, \label{eq:instability-result-combine}
\end{equation}
for some constants $C_{R}$ depending only on $R$ and absolute constant $c_{0}$. 
\end{thm}

This paper is organized as follows.  We list some preliminary materials
in Section~\ref{sec:Auxiliary-propositions}. Section~\ref{sec3} is a collection of some useful estimates needed in the proofs of main theorems. We then prove Theorem~\ref{thm:main} and Theorem~\ref{thm:second-main} in Section~\ref{sec4}. Like other works in the instability of the inverse problem, our proof is based on Kolmogorov's entropy theorem \cite{KT61}.

\section{\label{sec:Auxiliary-propositions}Preliminaries}

\subsection{Existence of $\theta$-discrete set for some neighborhood }

Fixing $r_{0}\in(0,1)$. Given any $\alpha>0$, $\theta>0$, and $\beta>0$,
we consider the following set 
\[
\mathcal{N}_{\alpha\beta}^{\theta}(B_{r_{0}}):=\begin{Bmatrix}\begin{array}{l|l}
f\ge 0 & {\rm supp}\,(f)\subset B_{r_{0}},\|f\|_{L^\infty}\le\theta,\|f\|_{\mathcal{C}^{\alpha}}\le\beta\end{array}\end{Bmatrix},
\]
where $B_{r_0}$ denotes the ball of radius $r_0$ centered at the origin. The following proposition can be found in \cite[Lemma 5.2]{ZZ19} (or in \cite{KT61} in a more abstract form),
see also \cite[Lemma 2]{Man01} for a direct proof for the special
case when $r_{0}=\frac{1}{2}$. 
\begin{prop}
\label{prop:Kolmogorov}There exists a constant $\mu>0$ such that
the following statement holds for all $\beta>0$ and for all $\theta\in(0,\mu\beta)$:
\[
\text{there exists a }\theta\text{-discrete }(\text{i.e. }\theta\text{-distinguishable})\text{ subset }Z\text{ of }(\mathcal{N}_{\alpha\beta}^{\theta}(B_{r_{0}}),\|\bullet\|_{L^{\infty}}),
\]
that is, $\|f_{1}-f_{2}\|_{L^{\infty}}\ge\theta$ for all $f_{1},f_{2}\in Z\subset\mathcal{N}_{\alpha\beta}^{\theta}(K)$.
Moreover, the cardinality of $Z$, denoted by $|Z|$, is bounded below
by 
\[
|Z|\ge\exp\bigg[2^{-(d+1)}\bigg(\frac{\mu\beta}{\theta}\bigg)^{\frac{d}{\alpha}}\bigg].
\]
\end{prop}

\subsection{Matrix representation via spherical harmonics}

As in \cite{Man01,ZZ19}, we will use  the set of $d$-dimensional spherical harmonics: 
\[
\mathbb{H}^{d}:=\begin{Bmatrix}\begin{array}{l|l}
Y_{mj} & m\ge 0,\, 1\le j\le p_{m}\end{array}\end{Bmatrix},
\]
where 
\begin{equation}
p_{m}:=\begin{pmatrix}m+d-1\\
d-1
\end{pmatrix}-\begin{pmatrix}m+d-3\\
d-1
\end{pmatrix}\le2(1+m)^{d-2},\label{eq:spherical-index}
\end{equation}
which is a complete orthogonal set in $L^{2}(\mathcal{S}^{d-1})$.
Indeed, $\|\bullet\|_{H^{s}(\mathcal{S}^{d-1})}$ is also equivalent
to the following norm: 
\begin{equation}
\bigg\|\sum_{m,j}a_{mj}Y_{mj}\bigg\|_{H^{s}(\mathcal{S}^{d-1})}^{2}:=\sum_{m,j}(1+m)^{2s}|a_{mj}|^{2},\label{eq:equiv}
\end{equation}
see e.g. \cite{Man01}. 

Given any bounded linear operator $\mathcal{A}:H^{s}(\mathcal{S}^{d-1})\rightarrow H^{-s}(\mathcal{S}^{d-1})$,
we define $a_{mjnk}:=\langle\mathcal{A}Y_{mj},Y_{nk}\rangle$ and consider
the Banach space:
\[
X_{s}:=\begin{Bmatrix}\begin{array}{l|l}
(a_{mjnk}) & \|(a_{mjnk})\|_{X_{s}}:={\displaystyle \sup_{mjnk}}(1+\max\{m,n\})^{\frac{d}{2}-s}|a_{mjnk}|\end{array}\end{Bmatrix}.
\]
(see \cite{Man01,ZZ19}). The following proposition can be found in \cite[Lemma 5.3]{ZZ19},
which is crucial in our work. 
\begin{prop}
\label{prop:matrix-repn}If $s\ge\frac{d}{2}$, then 
\[
\|\mathcal{A}\|_{s\rightarrow-s}\le4\sqrt{2}\|(a_{mjnk})\|_{X_{s}}.
\]
\end{prop}

Proposition~\ref{prop:matrix-repn} suggests that one can interpret $a_{mjnk}:=\langle\mathcal{A}Y_{mj},Y_{nk}\rangle$
as the matrix representation of the bounded linear operator $\mathcal{A}:H^{s}(\mathcal{S}^{d-1})\rightarrow H^{-s}(\mathcal{S}^{d-1})$. 

\section{\label{sec3}Some useful estimates}

In this section, we would like to derive some useful estimates that are needed in the proofs of main theorems.
\subsection{Elliptic estimate}
\begin{lem}\label{lem:elliptic-est}
Fixing $\lambda>0$. Let $\kappa^{2}>0$, $q\in L^{\infty}(B_{1})$,
$f\in L^{2}(B_{1})$, and $v\in H_{0}^{1}(B_{1})$ be a solution to
\begin{equation}
\bigg[\Delta+q(x)+\kappa^{2}\bigg]v=f\quad\text{in}\;\;B_{1}.\label{eq:sch1}
\end{equation}
If 
\begin{equation}\label{imaginary}
\Im(q)(x)\ge\lambda\;\, \mbox{or}\;\, \Im(q(x))\le-\lambda\;\;\, a.e., 
\end{equation}
then 
\begin{equation}
\|v\|_{L^{2}(B_{1})}\le\frac{1}{\lambda}\|f\|_{L^{2}(B_{1})}.\label{eq:ellip1}
\end{equation}
\end{lem}

\begin{proof}
Multiplying \eqref{eq:sch1} by $\overline{v}$ and integrating
over $B_{1}$, we have 
\begin{equation}
-\int_{B_{1}}|\nabla v|^{2}\,dx+\int_{B_{1}}q(x)|v|^{2}\,dx+\kappa^{2}\int_{B_{1}}|v|^{2}\,dx=\int_{B_{1}}f\overline{v}\,dx.\label{eq:sch-test}
\end{equation}
Taking the imaginary part of \eqref{eq:sch-test}, we have 
\[
\lambda\|v\|_{L^{2}(B_{1})}^{2} \le \bigg|\int_{B_{1}}\Im(q(x))|v|^{2}\,dx\bigg|=\bigg|\Im\int_{B_{1}}f\overline{v}\,dx\bigg|\le\frac{1}{2\lambda}\|f\|_{L^{2}(B_{1})}^{2}+\frac{\lambda}{2}\|v\|_{L^{2}(B_{1})}^{2},
\]
which implies \eqref{eq:ellip1}. 
\end{proof}
The following elliptic estimate is useful in our proof. 
\begin{prop}
\label{prop:ellip2}Let $\lambda>0$, $\kappa^{2}>0$, $q\in L^{\infty}(B_{1})$,
$\phi\in H^{\frac{3}{2}}(B_{1})$, and $u\in H^{1}(B_{1})$ be a solution
to 
\begin{equation}
\begin{cases}
\bigg[\Delta+q(x)+\kappa^{2}\bigg]u=0 & \text{in}\;\;B_{1},\\
u=\phi & \text{on}\;\;\mathcal{S}^{d-1}.
\end{cases}\label{eq:sch-reg}
\end{equation}
Assume that \eqref{imaginary} holds. Then there exists a positive constant $C_{d}$ such that 
\begin{equation}
\|u\|_{L^{2}(B_{1})}\le C_{d}\bigg(\frac{1}{\lambda}(1+\|q\|_{L^{\infty}(B_{1})}+\kappa^{2})+1\bigg)\|\phi\|_{H^{\frac{3}{2}}(\mathcal{S}^{d-1})}.\label{eq:elliptic-est}
\end{equation}
Here and after, $C_{d}$ (and $C_d', C_d''$) are general constants depending only on dimension $d$.
\end{prop}

\begin{proof} Let $\tilde\phi\in H^{2}(B_{1})$ be an extension of $\phi$ satisfying
the inequality 
\begin{equation}
\|\tilde\phi\|_{H^{2}(B_{1})}\le C_{d}\|\phi\|_{H^{\frac{3}{2}}(\mathcal{S}^{d-1})}\label{eq:trace}
\end{equation}
for some positive constant $C_{d}$. By letting $v=u-\tilde\phi\in H_{0}^{1}(B_{1})$, we have 
\begin{equation}
\bigg[\Delta+q(x)+\kappa^{2}\bigg]v=f\quad\text{in}\;\;\,B_{1}, \label{eq:elliptic-regularity1}
\end{equation}
where $f=-\bigg[\Delta+q(x)+\kappa^{2}\bigg]\tilde\phi\in L^{2}(B_{1})$.
From \eqref{eq:ellip1}, we have 
\[
\|v\|_{L^{2}(B_{1})}\le\frac{1}{\lambda}\|f\|_{L^{2}(B_{1})}\le\frac{(1+\|q\|_{L^{\infty}(B_{1})}+\kappa^{2})}{\lambda}\|\tilde\phi\|_{H^{2}(B_{1})}.
\]
Since $\|u\|_{L^{2}(B_{1})}\le\|v\|_{L^{2}(B_{1})}+\|\tilde\phi\|_{L^{2}(B_{1})}$,
together with \eqref{eq:trace}, we have 
\begin{align*}
\|u\|_{L^{2}(B_{1})} & \le\|v\|_{L^{2}(B_{1})}+\|\tilde\phi\|_{L^{2}(B_{1})}\\
 & \le\bigg(\frac{1}{\lambda}(1+\|q\|_{L^{\infty}(B_{1})}+\kappa^{2})+1\bigg)\|\tilde\phi\|_{H^{2}(B_{1})}\\
 & \le C_{d}\bigg(\frac{1}{\lambda}(1+\|q\|_{L^{\infty}(B_{1})}+\kappa^{2})+1\bigg)\|\phi\|_{H^{\frac{3}{2}}(\mathcal{S}^{d-1})},
\end{align*}
which implies \eqref{eq:elliptic-est}.
\end{proof}

\subsection{\label{sec:Matrix-repn}Matrix representation}

Let $q_{{\rm ref}}\in L^{\infty}(B_{1})$ satisfy \eqref{imaginary}. For $R>0$, and $0<r_{0}<1$,
we define 
\[
B_{+,R}^{\infty,r_{0}}:=\begin{Bmatrix}\begin{array}{l|l}
q\text{ is real-valued} & {\rm supp}\,(q)\subset\overline{B_{r_{0}}},0\le q\le R\;\, a.e.\end{array}\end{Bmatrix}.
\]
For each $\kappa^{2}>0$, we consider the mapping 
\[
\Gamma(q;q_{{\rm ref}}):=\Lambda_{q+q_{{\rm ref}}}-\Lambda_{q_{{\rm ref}}}\quad\text{for all }\;\,q\in B_{+,R}^{\infty,r_{0}}.
\]
Let $u_{mj}\in H^{1}(B_{1})$
be the unique solution to 
\[
\begin{cases}
(\Delta+q_{{\rm ref}}(x)+q(x)+\kappa^{2})u_{mj}=0 & \text{in}\;\;B_{1},\\
u_{mj}=Y_{mj} & \text{on}\;\;\mathcal{S}^{d-1}.
\end{cases}
\]
Similarly, let $\mathring{u}_{mj}\in H^{1}(B_{1})$ be the unique
solution to 
\[
\begin{cases}
(\Delta+q_{{\rm ref}}(x)+\kappa^{2})\mathring{u}_{mj}=0 & \text{in}\;\;B_{1},\\
\mathring{u}_{mj}=Y_{mj} & \text{on}\;\;\mathcal{S}^{d-1}.
\end{cases}
\]
We also define $v_{mj}\in H^{1}(B_{1})$ the unique solution to
\[
\begin{cases}
(\Delta+\overline{q_{{\rm ref}}(x)}+q(x)+\kappa^{2})v_{mj}=0 & \text{in}\;\;B_{1},\\
v_{mj}=Y_{mj} & \text{on}\;\;\mathcal{S}^{d-1},
\end{cases}
\]
and let $\mathring{v}_{mj}\in H^{1}(B_{1})$ be the unique solution
to 
\[
\begin{cases}
(\Delta+\overline{q_{{\rm ref}}(x)}+\kappa^{2})\mathring{v}_{mj}=0 & \text{in}\;\;B_{1},\\
\mathring{v}_{mj}=Y_{mj} & \text{on}\;\;\mathcal{S}^{d-1}.
\end{cases}
\]
By \eqref{eq:equiv} and \eqref{eq:elliptic-est}, we can derive 
\begin{subequations}
\begin{align}
\|u_{mj}\|_{L^{2}(B_{1})} & \le C_{d}'\bigg(\frac{1}{\lambda}(1+R+\lambda+\kappa^{2})+1\bigg)(1+m)^{\frac{3}{2}},\label{eq:umj-est1a}\\
\|\mathring{u}_{mj}\|_{L^{2}(B_{1})} & \le C_{d}'\bigg(\frac{1}{\lambda}(1+R+\lambda+\kappa^{2})+1\bigg)(1+m)^{\frac{3}{2}},\label{eq:umj-est1b}\\
\|v_{mj}\|_{L^{2}(B_{1})} & \le C_{d}'\bigg(\frac{1}{\lambda}(1+R+\lambda+\kappa^{2})+1\bigg)(1+m)^{\frac{3}{2}},\label{eq:vmj-est1c}\\
\|\mathring{v}_{mj}\|_{L^{2}(B_{1})} & \le C_{d}'\bigg(\frac{1}{\lambda}(1+R+\lambda+\kappa^{2})+1\bigg)(1+m)^{\frac{3}{2}}.\label{eq:vmj-est1d}
\end{align}
\end{subequations}

Write $\tilde{Y}_{mj}(x):=|x|^{m}Y_{mj}(x/|x|)$, which is harmonic in $B_{1}$. We can prove
\begin{lem}
\label{lem:matrix-repn}Let $q_{{\rm ref}}\in L^{\infty}(B_{1})$ satisfy \eqref{imaginary} and $q\in B_{+,R}^{\infty,r_{0}}$.
Then 
\begin{align}
\langle\Gamma(q;q_{{\rm ref}})Y_{mj},Y_{nk}\rangle & =I_{mjnk}^{(1)}(q)+I_{mjnk}^{(2)}(q_{{\rm ref}})\nonumber \\
 & =L_{mjnk}^{(1)}(q)+L_{mjnk}^{(2)}(q_{{\rm ref}}),\label{eq:matrix-repn}
\end{align}
where 
\begin{align*}
I_{mjnk}^{(1)}(q) & =-\int_{B_{r_{0}}}q(x)u_{mj}\overline{\tilde{Y}_{nk}}\,dx,\\
I_{mjnk}^{(2)}(q_{{\rm ref}}) & =-\int_{B_{1}}(q_{{\rm ref}}(x)+\kappa^{2})(u_{mj}-\mathring{u}_{mj})\overline{\tilde{Y}_{nk}}\,dx,\\
L_{mjnk}^{(1)}(q) & =-\int_{B_{r_{0}}}q(x)\overline{v_{nk}}\tilde{Y}_{mj}\,dx,\\
L_{mjnk}^{(2)}(q_{{\rm ref}}) & =-\int_{B_{1}}(\overline{q_{{\rm ref}}(x)}+\kappa^{2})\overline{(v_{nk}-\mathring{v}_{nk})}\tilde{Y}_{mj}\,dx.
\end{align*}
\end{lem}

\begin{proof}
We begin to prove 
\begin{equation}
\langle\Gamma(q;q_{{\rm ref}})Y_{mj},Y_{nk}\rangle=I_{mjnk}^{(1)}(q;q_{{\rm ref}})+I_{mjnk}^{(2)}(q;q_{{\rm ref}}).\label{eq:part1}
\end{equation}
Using integration by parts, we have 
\begin{align}
\langle\Gamma(q;q_{{\rm ref}})Y_{mj},Y_{nk}\rangle & =\int_{\mathcal{S}^{d-1}}\partial_{r}(u_{mj}-\mathring{u}_{mj})\overline{Y_{nk}}\,ds\nonumber \\
 & =\int_{B_{1}}\Delta(u_{mj}-\mathring{u}_{mj})\overline{\tilde{Y}_{nk}}\,dx.\label{eq:int-by-part}
\end{align}
In the computation of \eqref{eq:int-by-part}, we used the fact $u_{mj}=\mathring{u}_{mj}$
on $\mathcal{S}^{d-1}$ and $\tilde{Y}_{nk}$ is harmonic in $B_{1}$.
Since 
\[
(\Delta+q_{{\rm ref}}(x)+q(x)+\kappa^{2})u_{mj}=0\quad\text{and}\quad(\Delta+q_{{\rm ref}}(x)+\kappa^{2})\mathring{u}_{mj}=0,
\]
and ${\rm supp}\,(q)\subset B_{r_{0}}$, \eqref{eq:int-by-part} implies  
\begin{align*}
 & \langle\Gamma(q;q_{{\rm ref}})Y_{mj},Y_{nk}\rangle\\
 & =-\int_{B_{r_{0}}}q(x)u_{mj}\overline{\tilde{Y}_{nk}}\,dx-\int_{B_{1}}(q_{{\rm ref}}(x)+\kappa^{2})(u_{mj}-\mathring{u}_{mj})\overline{\tilde{Y}_{nk}}\,dx,
\end{align*}
which is exactly \eqref{eq:part1}. 

Next, we want to establish 
\begin{equation}
\langle\Gamma(q;q_{{\rm ref}})Y_{mj},Y_{nk}\rangle=L_{mjnk}^{(1)}(q;q_{{\rm ref}})+L_{mjnk}^{(2)}(q;q_{{\rm ref}}).\label{eq:part2}
\end{equation}
Since the adjoint operator of $\Gamma(q;q_{{\rm ref}})$ is $\Gamma(q;\overline{q_{{\rm ref}}})$,
similar to the derivation of \eqref{eq:int-by-part}, we have 
\begin{align*}
\langle\Gamma(q;q_{{\rm ref}})Y_{mj},Y_{nk}\rangle & =\langle Y_{mj},\Gamma({q};\overline{q_{{\rm ref}}})Y_{nk}\rangle\\
 & =\int_{B_{1}}\overline{\Delta(v_{nk}-\mathring{v}_{nk})}\tilde{Y}_{mj}\,dx.
\end{align*}
Since 
\[
(\Delta+\overline{q_{{\rm ref}}(x)}+q(x)+\kappa^{2})v_{mj}=0\quad\text{and}\quad(\Delta+\overline{q_{{\rm ref}}(x)}+\kappa^{2})\mathring{v}_{mj}=0,
\]
and ${\rm supp}\,(q)\subset B_{r_{0}}$, we obtain
\begin{align*}
 & \langle\Gamma(q;q_{{\rm ref}})Y_{mj},Y_{nk}\rangle=\langle Y_{mj},\Gamma({q};\overline{q_{{\rm ref}}})Y_{nk}\rangle\\
 & =-\int_{B_{r_{0}}}q(x)\overline{v_{nk}}\tilde{Y}_{mj}\,dx-\int_{B_{1}}(\overline{q_{{\rm ref}}(x)}+\kappa^{2})\overline{(v_{nk}-\mathring{v}_{nk})}\tilde{Y}_{mj}\,dx
\end{align*}
and hence \eqref{eq:part2}. 
\end{proof}
Next, we define 
\begin{equation}
M_{mjnk}^{(1)}(q):=\begin{cases}
I_{mjnk}^{(1)}(q) & \text{if}\,\,n\ge m,\\
L_{mjnk}^{(1)}(q) & \text{if}\,\,n<m,
\end{cases}\quad M_{mjnk}^{(2)}(q_{{\rm ref}}):=\begin{cases}
I_{mjnk}^{(2)}(q_{{\rm ref}}) & \text{if}\,\,n\ge m,\\
L_{mjnk}^{(2)}(q_{{\rm ref}}) & \text{if}\,\,n<m.
\end{cases} \label{eq:M-mjnk}
\end{equation}
Therefore, from \eqref{eq:matrix-repn}, we have 
\begin{equation}
|\langle\Gamma(q;q_{{\rm ref}})Y_{mj},Y_{nk}\rangle|\le|M_{mjnk}^{(1)}(q)|+|M_{mjnk}^{(2)}(q_{{\rm ref}})|.\label{eq:matrix-repn2}
\end{equation}

\begin{lem}
\label{lem:matrix-repn-est}Let $q_{{\rm ref}}\in L^{\infty}(B_{1})$ be such that \eqref{imaginary} holds. Let $\ell=\max\{m,n\}$,
$\kappa^{2}>0$, and $q\in B_{+,R}^{\infty,r_{0}}$. Then 
\begin{align}
|M_{mjnk}^{(1)}(q)| & \le C_{d}''\Phi(R,\lambda,\kappa)(1+\ell)r_{0}^{\ell},\label{eq:est-M1}\\
|M_{mjnk}^{(2)}(q_{{\rm ref}})| & \le C_{d}''\Phi(R,\lambda,\kappa)(\|q_{{\rm ref}}\|_{L^{\infty}(B_{1})}+\kappa^{2})(1+\ell)\label{eq:est-M2}
\end{align}
with $C_d''>1$, where 
\begin{equation}
\Phi(R,\lambda,\kappa)=(R+1)\bigg(\frac{1}{\lambda}(1+R+\lambda+\kappa^{2})+1\bigg).\label{eq:temp}
\end{equation}
\end{lem}

\begin{rem}\label{rem:observation1}
\label{rem:net}Denote $\tau:=s-\frac{d+2}{2}$. It follows from \eqref{eq:est-M1}
and \eqref{eq:est-M2} that 
\begin{align}
(1+\ell)^{\frac{d}{2}-s}|M_{mjnk}^{(1)}(q)| & \le C_{d}''\Phi(R,\lambda,\kappa)(1+\ell)^{-\tau}r_{0}^{\ell},\label{eq:est-M1-1}\\
(1+\ell)^{\frac{d}{2}-s}|M_{mjnk}^{(2)}(q_{{\rm ref}})| & \le C_{d}''\Phi(R,\lambda,\kappa)(\|q_{{\rm ref}}\|_{L^{\infty}(B_{1})}+\kappa^{2})(1+\ell)^{-\tau}.\label{eq:est-M2-1}
\end{align}
Therefore, if $s\ge\frac{d+2}{2}$ (i.e. $\tau \ge 0$), we know that 
\[
\|(\langle\Gamma(q;q_{{\rm ref}})Y_{mj},Y_{nk}\rangle)\|_{X_{s}}=\sup_{mjnk}(1+\ell)^{\frac{d}{2}-s}|\langle\Gamma(q;q_{{\rm ref}})Y_{mj},Y_{nk}\rangle|<\infty.
\]
If we define $\Gamma_{mjnk}^{q_{{\rm ref}}}(q):=\langle\Gamma(q;q_{{\rm ref}})Y_{mj},Y_{nk}\rangle$,
we have 
\begin{equation}
(\Gamma_{mjnk}^{q_{{\rm ref}}}(B_{+,R}^{\infty,r_{0}}))\subset X_{s}\quad\text{for all }s \ge \frac{d+2}{2}.\label{eq:net}
\end{equation}
\end{rem}

\begin{proof}
[Proof of Lemma~{\rm \ref{lem:matrix-repn-est}}] Like in \cite[(16)]{ZZ19}, using \eqref{eq:umj-est1a}  yields
\begin{align*}
|I_{mjnk}^{(1)}(q)| & =\bigg|\int_{B_{r_{0}}}q(x)|x|^{n}\overline{Y_{nk}(x/|x|)}u_{mj}(x)\,dx\bigg|\\
 & \le R\bigg(\int_{0}^{r_{0}}r^{2n+d-1}\,dr\bigg)^{\frac{1}{2}}\|u_{mj}\|_{L^{2}(B_{1})}\\
 & \le C_{d}'R\bigg(\frac{1}{\lambda}(1+R+\lambda+\kappa^{2})+1\bigg)\frac{(1+m)^{\frac{3}{2}}}{\sqrt{2n+d}}r_{0}^{n+\frac{d}{2}}\\
 & \le C_{d}'R\bigg(\frac{1}{\lambda}(1+R+\lambda+\kappa^{2})+1\bigg)\frac{(1+m)^{\frac{3}{2}}}{\sqrt{2n+d}}r_{0}^{n}.
\end{align*}
Similarly, from \eqref{eq:vmj-est1c}, we have 
\[
|L_{mjnk}^{(1)}(q)|\le C_{d}'R\bigg(\frac{1}{\lambda}(1+R+\lambda+\kappa^{2})+1\bigg)\frac{(1+n)^{\frac{3}{2}}}{\sqrt{2m+d}}r_{0}^{m}.
\]
Therefore, 
\begin{align*}
|M_{mjnk}^{(1)}(q)| & \le C_{d}'R\bigg(\frac{1}{\lambda}(1+R+\lambda+\kappa^{2})+1\bigg)\frac{(1+n)^{\frac{3}{2}}}{\sqrt{2n+d}}r_{0}^{n}\quad\text{if}\;\;n\ge m,\\
|M_{mjnk}^{(1)}(q)| & \le C_{d}'R\bigg(\frac{1}{\lambda}(1+R+\lambda+\kappa^{2})+1\bigg)\frac{(1+m)^{\frac{3}{2}}}{\sqrt{2m+d}}r_{0}^{m}\quad\text{if}\;\;n<m,
\end{align*}
which gives \eqref{eq:est-M1}. 

Now, we turn to \eqref{eq:est-M2}. Following \cite[(16)]{ZZ19} again,
combining \eqref{eq:umj-est1a} and \eqref{eq:umj-est1b} implies
\begin{align*}
|I_{mjnk}^{(2)}(q_{{\rm ref}})| & =\bigg|\int_{B_{1}}(q_{{\rm ref}}(x)+\kappa^{2})(u_{mj}-\mathring{u}_{mj})|x|^{n}\overline{Y_{nk}(x/|x|)}\,dx\bigg|\\
 & \le\bigg(\|q_{{\rm ref}}\|_{L^{\infty}(B_{1})}+\kappa^{2}\bigg)\bigg(\int_{0}^{1}r^{2n+d-1}\,dr\bigg)^{\frac{1}{2}}\bigg(\|u_{mj}\|_{L^{2}(B_{1})}+\|\mathring{u}_{mj}\|_{L^{2}(B_{1})}\bigg)\\
 & \le2C_{d}'\bigg(\|q_{{\rm ref}}\|_{L^{\infty}(B_{1})}+\kappa^{2}\bigg)\bigg(\frac{1}{\lambda}(1+R+\lambda+\kappa^{2})+1\bigg)\frac{(1+m)^{\frac{3}{2}}}{\sqrt{2n+d}}.
\end{align*}
Likewise, by \eqref{eq:vmj-est1c} and \eqref{eq:vmj-est1d}, we have 
\[
|L_{mjnk}^{(2)}(q_{{\rm ref}})|\le2C_{d}'\bigg(\|q_{{\rm ref}}\|_{L^{\infty}(B_{1})}+\kappa^{2}\bigg)\bigg(\frac{1}{\lambda}(1+R+\lambda+\kappa^{2})+1\bigg)\frac{(1+n)^{\frac{3}{2}}}{\sqrt{2m+d}}.
\]
Then we can derive \eqref{eq:est-M2} as above. 
\end{proof}

\subsection{Construction of a net}

In view of \eqref{eq:net} in Remark~\ref{rem:net}, we want to construct a $\delta$-net $Y$ of $((\Gamma_{mjnk}^{q_{{\rm ref}}}(B_{+,R}^{\infty,r_{0}})),\|\bullet\|_{X_{s}})$,
which is not too-large. Precisely, we want to obtain the following
proposition. 
\begin{prop}
\label{prop:net-const}Assume that $q_{{\rm ref}}\in L^{\infty}(B_{1})$ satisfies \eqref{imaginary}. Let $s>\frac{d+2}{2}$
and $\kappa^{2}>0$. Define $\tau:=s-\frac{d+2}{2}$. Given any $0<\delta<\Phi$, where $\Phi=\Phi(R,\lambda,\kappa)$
is the function given in \eqref{eq:temp}, there exists a $\delta$-net $Y$ of $((\Gamma_{mjnk}^{q_{{\rm ref}}}(B_{+,R}^{\infty,r_{0}})),\|\bullet\|_{X_{s}})$
such that 
\begin{equation}
\log|Y|\le\eta\bigg[1+\log\bigg(1+\frac{\Phi}{\delta}\bigg)+(\|q_{{\rm ref}}\|_{L^{\infty}(B_{1})}+\kappa^{2})\frac{\Phi}{\delta}+\bigg((\|q_{{\rm ref}}\|_{L^{\infty}(B_{1})}+\kappa^{2})\frac{\Phi}{\delta}\bigg)^{\frac{1}{\tau}}\bigg]^{2d}\label{eq:net-size}
\end{equation}
for some constant $\eta=\eta(d,s,r_{0})$. 
\end{prop}

\begin{rem}\label{fk}
In view of Remark~{\rm \ref{rem:observation1}}, here we exclude the case $s = \frac{d+2}{2}$ (i.e. $s > \frac{d+2}{2}$ ), to ensure that the quantity $1/\tau$ in \eqref{eq:net-size} is well-defined.
\end{rem}

\begin{proof}
\textbf{Step 1: Initialization.} Suggested by \eqref{eq:est-M1-1} and
\eqref{eq:est-M2-1}, let $\ell_{1}$ and $\ell_{2}$ be the solution to 
\[
(1+\ell_{1})^{-\tau}r_{0}^{\ell_{1}}=\frac{\delta}{8\sqrt{2}C_{d}''\Phi}\quad\text{and}\quad(1+\ell_{2})^{-\tau}(\|q_{{\rm ref}}\|_{L^{\infty}(B_{1})}+\kappa^{2})=\frac{\delta}{8\sqrt{2}C_{d}''\Phi},
\]
respectively. This is valid because $0<\delta<\Phi$ and $C_{d}''>1$. Let $\ell_{*}$ be the smallest nonnegative integer such that 
\begin{equation}
(1+\ell)^{-\tau}\bigg(r_{0}^{\ell}+\|q_{{\rm ref}}\|_{L^{\infty}(B_{1})}+\kappa^{2}\bigg)\le\frac{\delta}{4\sqrt{2}C_{d}''\Phi}\quad\text{for all }\;\;\ell\ge\ell_{*}.\label{eq:ellstar-def}
\end{equation}
Observe that 
\begin{align*}
 & (1+(\ell_{1}+\ell_{2}))^{-\tau}\bigg(r_{0}^{\ell_{1}+\ell_{2}}+\|q_{{\rm ref}}\|_{L^{\infty}(B_{1})}+\kappa^{2}\bigg)\\
 & =r_{0}^{\ell_{2}}(1+(\ell_{1}+\ell_{2}))^{-\tau}r_{0}^{\ell_{1}}+(1+(\ell_{1}+\ell_{2}))^{-\tau}\bigg(\|q_{{\rm ref}}\|_{L^{\infty}(B_{1})}+\kappa^{2}\bigg)\\
 & \le(1+\ell_{1})^{-\tau}r_{0}^{\ell_{1}}+(1+\ell_{2})^{-\tau}(\|q_{{\rm ref}}\|_{L^{\infty}(B_{1})}+\kappa^{2})\le\frac{\delta}{4\sqrt{2}C_{d}''\Phi},
\end{align*}
that is, $\ell_{1}+\ell_{2}$ satisfies \eqref{eq:ellstar-def}. Since
$\ell_{*}$ is the smallest nonnegative integer such that \eqref{eq:ellstar-def}
holds, then
\begin{equation}
\ell_{*}\le\ell_{1}+\ell_{2}.\label{eq:ell-star-1-2}
\end{equation}

Note that $\tau\log(1+\ell_{1})\ge0$. Thus, we have 
\[
\ell_{1}\log r_{0}\ge-\tau\log(1+\ell_{1})+\ell_{1}\log r_{0}=\log\bigg((1+\ell_{1})^{-\tau}r_{0}^{\ell_{1}}\bigg)=\log\bigg(\frac{\delta}{8\sqrt{2}C_{d}''\Phi}\bigg).
\]
Since $0<r_{0}<1$, i.e., $\log r_{0}<0$, it yields
\begin{equation}
\ell_{1}\le\frac{1}{\log r_{0}}\log\bigg(\frac{\delta}{8\sqrt{2}C_{d}''\Phi}\bigg)=\frac{1}{-\log r_{0}}\log\bigg(\frac{8\sqrt{2}C_{d}''\Phi}{\delta}\bigg).\label{eq:ell1}
\end{equation}
We also have 
\begin{equation}
\ell_{2}=\bigg(\frac{\delta}{8\sqrt{2}C_{d}''\Phi(\|q_{{\rm ref}}\|_{L^{\infty}(B_{1})}+\kappa^{2})}\bigg)^{-\frac{1}{\tau}}-1=(8\sqrt{2}C_{d}'')^{\frac{1}{\tau}}\bigg((\|q_{{\rm ref}}\|_{L^{\infty}(B_{1})}+\kappa^{2})\frac{\Phi}{\delta}\bigg)^{\frac{1}{\tau}}-1.\label{eq:ell2}
\end{equation}
Combining \eqref{eq:ell-star-1-2}, \eqref{eq:ell1}, and \eqref{eq:ell2}, leads to 
\begin{align}
\ell_{*}+1 & \le\frac{1}{-\log r_{0}}\log\bigg(\frac{8\sqrt{2}C_{d}''\Phi}{\delta}\bigg)+(8\sqrt{2}C_{d}'')^{\frac{1}{\tau}}\bigg((\|q_{{\rm ref}}\|_{L^{\infty}(B_{1})}+\kappa^{2})\frac{\Phi}{\delta}\bigg)^{\frac{1}{\tau}}\nonumber \\
 & =\frac{1}{-\log r_{0}}\bigg(\log(8\sqrt{2}C_{d}'')+\log\bigg(\frac{\Phi}{\delta}\bigg)\bigg)+(8\sqrt{2}C_{d}'')^{\frac{1}{\tau}}\bigg((\|q_{{\rm ref}}\|_{L^{\infty}(B_{1})}+\kappa^{2})\frac{\Phi}{\delta}\bigg)^{\frac{1}{\tau}}\nonumber \\
 & \le C_{1}\bigg(1+\log\bigg(1+\frac{\Phi}{\delta}\bigg)+\bigg((\|q_{{\rm ref}}\|_{L^{\infty}(B_{1})}+\kappa^{2})\frac{\Phi}{\delta}\bigg)^{\frac{1}{\tau}}\bigg).\label{eq:ellstar-est}
\end{align}
for some positive constant $C_{1}=C_{1}(d,r_{0},s)$. 

\textbf{Step 2: Construction of sets.} For each 4-tuple $(m,j,n,k)$
with $\ell:=\max\{m,n\}\le\ell_{*}$, \eqref{eq:est-M1} and \eqref{eq:est-M2}
imply 
\begin{align*}
|M_{mjnk}^{(1)}(q)| & \le C_{d}''C_{2}\Phi(R,\lambda,\kappa),\\
|M_{mjnk}^{(2)}(q_{{\rm ref}})| & \le C_{d}''\Phi(R,\lambda,\kappa)(\|q_{{\rm ref}}\|_{L^{\infty}(B_{1})}+\kappa^{2})(1+\ell_{*}),
\end{align*}
where $C_{2}=C_{2}(d,r_{0})={\displaystyle \sup_{\ell}}(1+\ell)r_{0}^{\ell}<\infty$.
Let $\delta'=\frac{\delta}{8\sqrt{2}}$ and define 
\begin{align*}
Y_{1}' & :=\begin{Bmatrix}\begin{array}{l|l}
a=a_{1}+ia_{2}\in \frac{\delta'}{\sqrt{2}}\mathbb{Z} + i \frac{\delta'}{\sqrt{2}}\mathbb{Z} & |a_{1}|,|a_{2}|\le C_{d}''C_{2}\Phi(R,\lambda,\kappa)\end{array}\end{Bmatrix},\\
Y_{2}' & :=\begin{Bmatrix}\begin{array}{l|l}
a=a_{1}+ia_{2}\in \frac{\delta'}{\sqrt{2}}\mathbb{Z} + i \frac{\delta'}{\sqrt{2}}\mathbb{Z} & |a_{1}|,|a_{2}|\le C_{d}''\Phi(R,\lambda,\kappa)(\|q_{{\rm ref}}\|_{L^{\infty}(B_{1})}+\kappa^{2})(1+\ell_{*})\end{array}\end{Bmatrix},
\end{align*}
where $\mathbb{Z}$ denotes the set of integers. Then, 
\begin{subequations}
\begin{align}
|Y_{1}'| &  = \bigg( 1+2\bigg\lfloor\frac{C_{d}''C_{2}\Phi}{\delta'/\sqrt{2}}\bigg\rfloor \bigg)^{2} \le \bigg( 1+\frac{32 C_{d}''C_{2}\Phi}{\delta} \bigg)^{2},\label{eq:Y1prime-card}\\
|Y_{2}'| & \le \bigg( 1+\frac{32 C_{d}''\Phi(\|q_{{\rm ref}}\|_{L^{\infty}(B_{1})}+\kappa^{2})(1+\ell_{*})}{\delta} \bigg)^{2}.\label{eq:Y2prime-card}
\end{align}
\end{subequations}
Define 
\begin{align*}
Y_{1} & :=\begin{Bmatrix}\begin{array}{l|l}
(b_{mjnk}) & \begin{array}{l}
b_{mjnk}\in Y'_{1},\;\text{ if }\max\{m,n\}\le\ell_{*};\\
b_{mjnk}=0, \;\text{ otherwise}
\end{array}\end{array}\end{Bmatrix},\\
Y_{2} & :=\begin{Bmatrix}\begin{array}{l|l}
(c_{mjnk}) & \begin{array}{l}
c_{mjnk}\in Y_{2}',\;\text{ if }\max\{m,n\}\le\ell_{*};\\
c_{mjnk}=0,\;\text{ otherwise}
\end{array}\end{array}\end{Bmatrix},
\end{align*}
and let $Y:=Y_{1}+Y_{2}$. 

\textbf{Step 3: Verifying $Y$ is a $\delta$-net of $((\Gamma_{mjnk}^{q_{{\rm ref}}}(B_{+,R}^{\infty,r_{0}})),\|\bullet\|_{X_{s}})$.}
Our goal is to construct 
\[
\begin{cases}
(b_{mjnk})\in Y_{1} & \text{which is an approximation of }(M_{mjnk}^{(1)}(q)),\\
(c_{mjnk})\in Y_{2} & \text{which is an approximation of }(M_{mjnk}^{(2)}(q_{{\rm ref}})).
\end{cases}
\]
\begin{itemize}
\item If $\max\{m,n\}\le\ell_{*}$, we take 
\[
\begin{cases}
b_{mjnk}'\in Y_{1}' & \text{the closest element to }M_{mjnk}^{(1)}(q),\\
c_{mjnk}'\in Y_{2}' & \text{the closest element to }M_{mjnk}^{(2)}(q_{{\rm ref}}).
\end{cases}
\]
Then 
\[
|b_{mjnk}'-M_{mjnk}^{(1)}(q)|\le\delta'\quad\text{and}\quad|c_{mjnk}'-M_{mjnk}^{(2)}(q_{{\rm ref}})|\le\delta'.
\]
By $s>\frac{d}{2}$, we then have 
\begin{align*}
 & 4\sqrt{2}(1+\max\{m,n\})^{\frac{d}{2}-s}(|b_{mjnk}'-M_{mjnk}^{(1)}(q)|+|c_{mjnk}'-M_{mjnk}^{(2)}(q_{{\rm ref}})|)\\
 & \le8\sqrt{2}\delta'=\delta.
\end{align*}
\item If $\ell=\max\{m,n\}>\ell_{*}$, we simply choose $b_{mjnk}'=c_{mjnk}'=0$.
Hence, we know that 
\begin{align*}
 & 4\sqrt{2}(1+\max\{m,n\})^{\frac{d}{2}-s}(|b_{mjnk}'-M_{mjnk}^{(1)}(q)|+|c_{mjnk}'-M_{mjnk}^{(2)}(q_{{\rm ref}})|)\\
 & =4\sqrt{2}(1+\ell)^{\frac{d}{2}-s}(|M_{mjnk}^{(1)}(q)|+|M_{mjnk}^{(2)}(q_{{\rm ref}})|)\\
 & \le4\sqrt{2}C_{d}''\Phi(1+\ell)^{-\tau}\bigg(r_{0}^{\ell}+\|q_{{\rm ref}}\|_{L^{\infty}(B_{1})}+\kappa^{2}\bigg)\quad(\text{by}\;\eqref{eq:est-M1-1}, \eqref{eq:est-M2-1})\\
 & \le\delta\quad(\text{using}\;\eqref{eq:ellstar-def}).
\end{align*}
\end{itemize}
Combining these two cases, we conclude that $Y$ is a $\delta$-net
of $((\Gamma_{mjnk}^{q_{{\rm ref}}}(B_{+,R}^{\infty,r_{0}})),\|\bullet\|_{X_{s}})$

\textbf{Step 4: Calculating the cardinality of $Y$.} Let $n_{\ell}$ be
the number of 4-tuples $(m,j,n,k)$ be such that $\max\{m,n\}=\ell$,
and let $n_{*}:=\sum_{\ell=0}^{\ell_{*}}n_{\ell}$. It is not hard to check that 
\[
|Y|=|Y_{1}||Y_{2}|=(|Y_{1}'||Y_{2}'|)^{n_{*}}.
\]
Also, by \eqref{eq:spherical-index}, we have 
\[
n_{*}\le8(1+\ell_{*})^{2d-2}.
\]
Therefore, from \eqref{eq:Y1prime-card}
and \eqref{eq:Y2prime-card}, it follows 
\begin{align*}
\log|Y| & \le8(1+\ell_{*})^{2d-2}\bigg[ 2\log\bigg(1+\frac{32C_{d}''C_{2}\Phi}{\delta}\bigg)\\
 & \quad\qquad+2\log\bigg(1+\frac{32C_{d}''\Phi(\|q_{{\rm ref}}\|_{L^{\infty}(B_{1})}+\kappa^{2})(1+\ell_{*})}{\delta}\bigg)\bigg].
\end{align*} 
Since $\log(1+t)\le t$ for all $t\ge0$, we obtain
\begin{align}
\log|Y| & \le 16(1+\ell_{*})^{2d-2}\bigg[\log\bigg(1+\frac{32C_{d}''C_{2}\Phi}{\delta}\bigg)+\frac{32C_{d}''\Phi(\|q_{{\rm ref}}\|_{L^{\infty}(B_{1})}+\kappa^{2})(1+\ell_{*})}{\delta}\bigg]\nonumber \\
 & \le 16(1+\ell_{*})^{2d-1}\bigg[\log\bigg(1+\frac{32C_{d}''C_{2}\Phi}{\delta}\bigg)+\frac{32C_{d}''\Phi(\|q_{{\rm ref}}\|_{L^{\infty}(B_{1})}+\kappa^{2})}{\delta}\bigg]\nonumber \\
 & \le\eta'(1+\ell_{*})^{2d-1}\bigg[1+\log\bigg(1+\frac{\Phi}{\delta}\bigg)+(\|q_{{\rm ref}}\|_{L^{\infty}(B_{1})}+\kappa^{2})\frac{\Phi}{\delta}\bigg]\label{eq:Y-card1}
\end{align}
for some constant $\eta'=\eta'(d,s,r_{0})$. Combining \eqref{eq:Y-card1}
with \eqref{eq:ellstar-est} yields 
\begin{align*}
\log|Y| & \le\eta'\bigg[C_{1}\bigg(1+\log\bigg(1+\frac{\Phi}{\delta}\bigg)+\bigg((\|q_{{\rm ref}}\|_{L^{\infty}(B_{1})}+\kappa^{2})\frac{\Phi}{\delta}\bigg)^{\frac{1}{\tau}}\bigg)\bigg]^{2d-1}\times\\
 & \quad\qquad\times\bigg[1+\log\bigg(1+\frac{\Phi}{\delta}\bigg)+(\|q_{{\rm ref}}\|_{L^{\infty}(B_{1})}+\kappa^{2})\frac{\Phi}{\delta}\bigg],
\end{align*}
which implies \eqref{eq:net-size}. 
\end{proof}

\section{Proofs of the main results}\label{sec4}

In this section, we would like to present detailed proofs of Theorem~\ref{thm:main} and Theorem~\ref{thm:second-main}. Before proving our results, we first give some observations. Taking $q_{{\rm ref}}=i$
(hence $\lambda=1$) and $s=\frac{d+4}{2}$ (i.e. $\tau=1$), see also Remark~{\rm \ref{rem:lambda-choice}}, we
see that \eqref{eq:net-size} becomes 
\begin{equation}
	\log|Y|\le\eta\bigg[1+\log\bigg(1+\frac{\Phi}{\delta}\bigg)+2(1+\kappa^{2})\frac{\Phi}{\delta}\bigg]^{2d},\label{eq:net-size-q-tau}
\end{equation}
where 
\begin{equation}
	\Phi(R,\kappa)=(R+1)(3+R+\kappa^{2}).\label{eq:phi-special}
\end{equation}
Note that 
\[
\inf_{0<\delta<\Phi}\frac{1}{(1+\kappa^{2})}\bigg[\log\bigg(1+\frac{\Phi}{\delta}\bigg)+2(1+\kappa^{2})\frac{\Phi}{\delta}\bigg]=\frac{\log2}{1+\kappa^{2}}+2<2+\log2
\]
for all $\kappa^{2}>0$. Therefore, given any $\theta$ satisfies
\begin{equation}
	2+\log2<\theta^{-\frac{1}{2\alpha}}<\infty\quad\bigg(\text{equivalently,}\quad\theta\in(0,(2+\log2)^{-2\alpha})\bigg),\label{eq:pert-observ}
\end{equation}
there exists a unique $\delta\in(0,\Phi)$ such that 

\begin{equation}
	\theta^{-\frac{1}{2\alpha}}=\frac{1}{1+\kappa^{2}}\bigg[\log\bigg(1+\frac{\Phi}{\delta}\bigg)+2(1+\kappa^{2})\frac{\Phi}{\delta}\bigg],\label{eq:equation-theta-delta}
\end{equation}
and \eqref{eq:net-size-q-tau} is reduced to 
\begin{equation}
	|Y|\le\exp\bigg[\eta\bigg(1+(1+\kappa^{2})\theta^{-\frac{1}{2\alpha}}\bigg)^{2d}\bigg]\label{eq:net-size-q-tau1}
\end{equation}

In view of these observations, we are ready to prove Theorem~\ref{thm:main}, the case of higher frequency. 
\begin{proof}
	[Proof of Theorem~{\rm \ref{thm:main}}] By Proposition~\ref{prop:Kolmogorov},
	there exists a constant $\mu>0$ such that the following statement
	holds for all $\beta>0$ and for all $\theta\in(0,\mu\beta)$: 
	\[
	\text{there exists a }\theta\text{-discrete subset }Z\text{ of }(\mathcal{N}_{\alpha\beta}^{\theta}(B_{r_{0}}),\|\bullet\|_{L^{\infty}})
	\]
	with 
	\[
	|Z|\ge\exp\bigg[2^{-(d+1)}\bigg(\frac{\mu\beta}{\theta}\bigg)^{\frac{d}{\alpha}}\bigg].
	\]
	It is suffices to restrict $0<\theta<\min\{(2+\log2)^{-2\alpha},R,\mu\beta\}$. Therefore, we can choose $0<\delta<\Phi$ such that \eqref{eq:equation-theta-delta} is satisfied.
	
	Let $0<r_{0}<1$. Since $0<\theta<R$, we see that $Z\subset\mathcal{N}_{\alpha\beta}^{\theta}(B_{r_{0}})\subset B_{+,R}^{\infty,r_{0}}$.
	Recall that we have chosen $q_{\rm ref}=i$. By Proposition~\ref{prop:net-const}, there exists a $\delta$-net
	$Y$ of $((\Gamma_{mjnk}^{i}(B_{+,R}^{\infty,r_{0}})),\|\bullet\|_{X_{s}})$
	with $s=\frac{d+4}{2}$ such that \eqref{eq:net-size-q-tau1} holds.
	Of course, $Y$ is also a $\delta$-net of $((\Gamma_{mjnk}^{i}(Z)),\|\bullet\|_{X_{s}})$. 
	
	We now choose $\beta>R$ sufficiently large (depends on $d,\mu,\theta,\kappa$) satisfying
	\[
	|Z|\ge\exp\bigg[2^{-(d+1)}\bigg(\frac{\mu\beta}{\theta}\bigg)^{\frac{d}{\alpha}}\bigg]>\exp\bigg[\eta\bigg(1+(1+\kappa^{2})\theta^{-\frac{1}{2\alpha}}\bigg)^{2d}\bigg]\ge|Y|.
	\]
	Therefore, we can choose
	two different $\tilde{q}_{1},\tilde{q}_{2}\in Z\subset\mathcal{N}_{\alpha\beta}^{\theta}(B_{r_{0}})\subset L^{\infty}(B_{1})$
	such that there exists $y\in Y$ such that 
	\begin{align*}
		\|(\Gamma^i_{mjnk}(\tilde{q}_{1})-y_{mjnk})\|_{X_{s}} & \le\delta,\\
		\|(\Gamma^i_{mjnk}(\tilde{q}_{2})-y_{mjnk})\|_{X_{s}} & \le\delta.
	\end{align*}
	Combining with Proposition~\ref{prop:matrix-repn},
	we conclude that 
	\[
	\|\Lambda_{q_{1}}-\Lambda_{q_{2}}\|_{\frac{d+4}{2}\rightarrow-\frac{d+4}{2}}\le8\sqrt{2}\delta,
	\]
	where $q_{1}=\tilde{q}_{1}+i$ and $q_{2}=\tilde{q}_{2}+i$. Now we
	estimate $\delta$ by $\theta$. 
	\begin{itemize}
		\item \textbf{Case 1.} If $(1+\kappa^{2})\frac{\Phi}{\delta}\le\log(1+\frac{\Phi}{\delta})$,
		then \eqref{eq:equation-theta-delta} implies 
		\[
		\theta^{-\frac{1}{2\alpha}}\le\frac{3}{1+\kappa^{2}}\log\bigg(1+\frac{\Phi}{\delta}\bigg),
		\]
		that is, 
		\[
		\delta\le(R+1)(3+R+\kappa^{2})\exp\bigg(-\frac{1+\kappa^{2}}{3}\theta^{-\frac{1}{2\alpha}}\bigg).
		\]
		\item \textbf{Case 2.} If $(1+\kappa^{2})\frac{\Phi}{\delta}\ge\log(1+\frac{\Phi}{\delta})$,
		the \eqref{eq:equation-theta-delta} implies 
		\[
		\theta^{-\frac{1}{2\alpha}}\le\frac{1}{1+\kappa^{2}}\bigg[3(1+\kappa^{2})\frac{\Phi}{\delta}\bigg]=\frac{3\Phi}{\delta}=3(R+1)(3+R+\kappa^{2})\delta^{-1},
		\]
		that is, 
		\[
		\delta\le3(R+1)(3+R+\kappa^{2})\theta^{\frac{1}{2\alpha}}.
		\]
	\end{itemize}
	Combining these two cases yields
	\[
	\delta\le(R+1)(3+R+\kappa^{2})\bigg[\exp\bigg(-\frac{1+\kappa^{2}}{3}\theta^{-\frac{1}{2\alpha}}\bigg)+3\theta^{\frac{1}{2\alpha}}\bigg],
	\]
	which gives \eqref{eq:DN-map-approx}. Since $Z$ is $\theta$-discrete
	subset with respect to the norm $\|\bullet\|_{L^{\infty}(B_{1})}$,
	 \eqref{eq:discrete-q1-q2} follows immediately, which is our desired result. 
\end{proof}

Now, we want to prove Theorem~\ref{thm:second-main} in which the frequency is small. First of all, we improve Proposition~\ref{prop:ellip2} without the
restriction \eqref{imaginary}. 
\begin{prop}
	Let $\kappa^{2}>0$, $q\in L^{\infty}(B_{1})$, $\phi\in H^{\frac{3}{2}}(B_{1})$,
	and $u\in H^{1}(B_{1})$ be the solution to \eqref{eq:sch-reg}. If
	\begin{equation}
		\|q\|_{L^{\infty}(B_{1})}\le\frac{1}{4}\kappa_{1}\quad\text{and}\quad\kappa^{2}\le\frac{1}{4}\kappa_{1},\label{eq:smallness-assump}
	\end{equation}
	where $\kappa_{1}$ is the first Dirichlet eigenvalue of $-\Delta$
	on $B_{1}$ (which depends only on dimension $d$), then
	\begin{equation}
		\|u\|_{L^{2}(B_{1})}\le C_d\|\phi\|_{H^{\frac{3}{2}}(\mathcal{S}^{d-1})}\label{eq:reg-improvement}
	\end{equation}
	for some constant $C_d$ depending only on dimension $d$. 
\end{prop}

\begin{proof}
	Let $\tilde{\phi}\in H^{2}(B_{1})$ as in \eqref{eq:trace} and
	consider $v=u-\tilde{\phi}\in H_{0}^{1}(B_{1})$. In this case, we
	write \eqref{eq:elliptic-regularity1} as 
	\begin{equation}
	\Delta v=f-q(x)v+\kappa^{2}v,\label{eq:elliptic-regularity1-modify}
	\end{equation}
	where $f=-(\Delta+q(x)+\kappa^{2})\tilde\phi$. Multiplying \eqref{eq:elliptic-regularity1-modify} by $-\overline{v}$
	and using the integration by parts, we have 
	\begin{align*}
		\int_{B_{1}}|\nabla v|^{2}\,dx & =\int_{B_{1}}f\overline{v}\,dx+\int_{B_{1}}q(x)|v|^{2}\,dx+\kappa^{2}\int_{B_{1}}|v|^{2}\,dx\\
		& \le\frac{1}{4\epsilon}\|f\|_{L^{2}(B_{1})}^{2}+\bigg(\epsilon+\|q\|_{L^{\infty}(B_{1})}+\kappa^{2}\bigg)\|v\|_{L^{2}(B_1)}^{2}
	\end{align*}
	for $\epsilon>0$. Poincar\'{e}'s inequality implies that
	we have $\kappa_{1}\|v\|_{L^{2}(B_{1})}^{2}\le\|\nabla v\|_{L^{2}(B_{1})}^{2}$ for all $v\in H_0^1(B_1)$.
	From \eqref{eq:smallness-assump}, it follows that
	\begin{align*}
		\kappa_{1}\|v\|_{L^{2}(B_{1})}^{2} & \le\frac{1}{4\epsilon}\|f\|_{L^{2}(B_{1})}^{2}+\bigg(\epsilon+\|q\|_{L^{\infty}(B_{1})}+\kappa^{2}\bigg)\|v\|_{L^{2}(B_1)}^{2}\\
		& \le\frac{1}{4\epsilon}\|f\|_{L^{2}(B_{1})}^{2}+\bigg(\epsilon+\frac{1}{2}\kappa_{1}\bigg)\|v\|_{L^{2}(B_1)}^{2}.
	\end{align*}
	Choosing $\epsilon=\frac{1}{4}\kappa_{1}$, we have 
	\[
	\|v\|_{L^{2}(B_{1})}^{2}\le\frac{4}{\kappa_{1}^{2}}\|f\|_{L^{2}(B_{1})}^{2}.
	\]
	Using \eqref{eq:smallness-assump} again gives
	\begin{align*}
		\|v\|_{L^{2}(B_{1})} & \le 2\kappa_{1}^{-1} \|f\|_{L^{2}(B_{1})}\le 2\kappa_{1}^{-1}(1+\|q\|_{L^{\infty}(B_{1})}+\kappa^{2})\|\tilde{\phi}\|_{H^{2}(B_{1})}\\
		& \le2\kappa_{1}^{-1}\bigg(1+\frac{1}{2}\kappa_{1}\bigg)\|\tilde{\phi}\|_{H^{2}(B_{1})}.
	\end{align*}
	Combining $\|u\|_{L^{2}(B_{1})}\le\|v\|_{L^{2}(B_{1})}+\|\tilde{\phi}\|_{L^{2}(B_{1})}$ and \eqref{eq:trace} implies \eqref{eq:reg-improvement}. 
\end{proof}

Let $u_{mj}\in H^{1}(B_{1})$, $\mathring{u}_{mj}\in H^{1}(B_{1})$,
$v_{mj}\in H^{1}(B_{1})$, $\mathring{v}_{mj}\in H^{1}(B_{1})$ as
in Subsection~\ref{sec:Matrix-repn}. In this case, by using \eqref{eq:reg-improvement}
rather than \eqref{eq:elliptic-est}, we can improve \eqref{eq:umj-est1a},
\eqref{eq:umj-est1b}, \eqref{eq:vmj-est1c}, and \eqref{eq:vmj-est1d}
as follows: 
\begin{subequations}
	\begin{align}
		\|u_{mj}\|_{L^{2}(B_{1})} & \le C_{d}'(1+m)^{\frac{3}{2}},\label{eq:umj-est2a}\\
		\|\mathring{u}_{mj}\|_{L^{2}(B_{1})} & \le C_{d}'(1+m)^{\frac{3}{2}},\label{eq:umj-est2b}\\
		\|v_{mj}\|_{L^{2}(B_{1})} & \le C_{d}'(1+m)^{\frac{3}{2}},\label{eq:vmj-est2c}\\
		\|\mathring{v}_{mj}\|_{L^{2}(B_{1})} & \le C_{d}'(1+m)^{\frac{3}{2}},\label{eq:vmj-est2d}
	\end{align}
\end{subequations}
for some constant $C_{d}'$. 

Now, choosing $q_{{\rm ref}}\equiv0$ in Lemma~\ref{lem:matrix-repn}, we can prove the following lemma. 
\begin{lem}
	For each $q\in B_{+,R}^{\infty,r_{0}}$, we have 
	\begin{align*}
		\langle\Gamma(q;0)Y_{mj},Y_{nk}\rangle & =I_{mjnk}^{(1)}(q)+I_{mjnk}^{(2)}(0),\\
		& =L_{mjnk}^{(1)}(q)+L_{mjnk}^{(2)}(0),
	\end{align*}
	where 
	\begin{align*}
		I_{mjnk}^{(1)}(q) & =-\int_{B_{r_{0}}}q(x)u_{mj}\overline{\tilde{Y}_{nk}}\,dx,\\
		I_{mjnk}^{(2)}(0) & =-\kappa^{2}\int_{B_{r_{0}}}(u_{mj}-\mathring{u}_{mj})\overline{\tilde{Y}_{nk}}\,dx,\\
		L_{mjnk}^{(1)}(q) & =-\int_{B_{r_{0}}}q(x)\overline{v_{nk}}\tilde{Y}_{mj}\,dx,\\
		L_{mjnk}^{(2)}(0) & =-\kappa^{2}\int_{B_{r_{0}}}\overline{(v_{nk}-\mathring{v}_{nk})}\tilde{Y}_{mj}\,dx.
	\end{align*}
\end{lem}

\begin{rem}
	Since the potential $q$ is real-valued,  $u_{mj}=v_{mj}$ and
	$\mathring{u}_{mj}=\mathring{v}_{mj}$. 
\end{rem}

Let $M_{mjnk}^{(l)}$ ($l=1,2$) be defined in \eqref{eq:M-mjnk}, then we
again obtain \eqref{eq:matrix-repn2}, that is, 
\[
|\langle\Gamma(q;0)Y_{mj},Y_{nk}\rangle|\le|M_{mjnk}^{(1)}(q)|+|M_{mjnk}^{(2)}(0)|.
\]
Following exactly the same arguments in the proof of Lemma~\ref{lem:matrix-repn-est},
we have 
\[
\begin{cases}
	\begin{array}{l}
		|M_{mjnk}^{(1)}(q)|\le C_{d}''(1+\ell)r_{0}^{\ell}\\
		|M_{mjnk}^{(1)}(q)|\le C_{d}''\kappa^{2}(1+\ell)
	\end{array} & \text{for all}\;\;q\in B_{+,\frac{1}{4}\kappa_{1}^{2}}^{\infty,r_{0}}.\end{cases}
\]
Consequently, we can replace the function $\Phi$ in \eqref{eq:est-M1} and \eqref{eq:est-M2} by the constant $1$. The observation significantly improves our estimates later.

Here, we remark that due to \eqref{eq:smallness-assump}, we choose $R=\frac{1}{4}\kappa_{1}$. Repeating the
same arguments as in the proof of Proposition~\ref{prop:net-const}, we can prove 
\begin{prop}
	\label{prop:net-improve}Let $s>\frac{d+2}{2}$ and $0<\kappa^{2}\le\frac{1}{4}\kappa_{1}$.
	Define $\tau=s-\frac{d+2}{2}$. Given any $0<\delta<1$, there exists
	a $\delta$-net $Y$ of $((\Gamma_{mjnk}^{0}(B_{+,\frac{1}{4}\kappa_{1}}^{\infty,r_{0}})),\|\bullet\|_{X_{s}})$
	with
	\begin{equation}
		\log|Y|\le\eta\bigg[1+\log\bigg(1+\frac{1}{\delta}\bigg)+\frac{\kappa^{2}}{\delta}+\bigg(\frac{\kappa^{2}}{\delta}\bigg)^{\frac{1}{\tau}}\bigg]^{2d}\label{eq:net-improve}
	\end{equation}
	for some constant $\eta=\eta(d,s,r_{0})$. 
\end{prop}

As above, we choose $s=\frac{d+4}{2}$, i.e. $\tau=1$, and \eqref{eq:net-improve}
becomes 
\begin{equation}
	\log|Y|\le\eta\bigg[1+\log\bigg(1+\frac{1}{\delta}\bigg)+\frac{2\kappa^{2}}{\delta}\bigg]^{2d}.\label{eq:net-improve-1}
\end{equation}
Observe that 
\[
\inf_{0<\delta<1}\bigg[\log\bigg(1+\frac{1}{\delta}\bigg)+\frac{2\kappa^{2}}{\delta}\bigg]\le\log2+\frac{\kappa_1}{2}
\]
Therefore, given any $\theta$ satisfies
\[
0<\theta<\left(\log2+\frac{\kappa_1}{2}\right)^{-2\alpha},
\]
there exists a unique $\tilde{\delta}\in(0,1)$ such that 
\begin{equation}
	\theta^{-\frac{1}{2\alpha}}=\log\bigg(1+\frac{1}{\tilde{\delta}}\bigg)+\frac{2\kappa^{2}}{\tilde{\delta}}.\label{eq:equation-theta-delta-improve}
\end{equation}
We see that \eqref{eq:equation-theta-delta-improve} is similar to
\eqref{eq:equation-theta-delta}, except that the multiplier $(1+\kappa^{2})^{-1}$
in front of the right-hand-side of \eqref{eq:equation-theta-delta} is removed. In view of \eqref{eq:equation-theta-delta-improve}, \eqref{eq:net-improve-1}
is reduced to 
\begin{equation}
	\log|Y|\le\eta(1+\theta^{-\frac{1}{2\alpha}})^{2d}.\label{eq:net-improve-2}
\end{equation}
Based on these observations, we are ready to prove Theorem~\ref{thm:second-main}. 
\begin{proof}
	[Proof of Theorem~{\rm \ref{thm:second-main}}] By Proposition~\ref{prop:Kolmogorov},
	there exists a constant $\mu>0$ such that for all $\beta>0$ and for all $\theta\in(0,\mu\beta)$: 
	\[
	\text{there exists a }\theta\text{-discrete subset }Z\text{ of }(\mathcal{N}_{\alpha,\beta}^{\theta}(B_{r_{0}}),\|\bullet\|_{L^{\infty}})
	\]
	with 
	\[
	|Z|\ge\exp\bigg[2^{-(d+1)}\bigg(\frac{\mu\beta}{\theta}\bigg)^{\frac{d}{\alpha}}\bigg].
	\]
	It suffices to restrict $0<\theta<\min\{(\frac 12{\kappa_1}+\log2)^{-2\alpha},\frac{1}{4}\kappa_{1},\mu\beta\}$.
	Therefore, we can find $0<\tilde{\delta}<1$ such that \eqref{eq:equation-theta-delta-improve}
	is satisfied. Since $0<\theta<\frac{1}{4}\kappa_{1}$, it is clear that $Z\subset\mathcal{N}_{\alpha\beta}^{\theta}(B_{r_{0}})\subset B_{+,\frac{1}{4}\kappa_{1}}^{\infty,r_{0}}$. Recall that we have chosen $q_{{\rm ref}}=0$. From Proposition~\ref{prop:net-improve},
	there exists a $\tilde{\delta}$-net $Y$ of $((\Gamma_{mjnk}^{0}(B_{+,\frac{1}{4}\kappa_{1}}^{\infty,r_{0}})),\|\bullet\|_{X_{s}})$
	with $s=\frac{d+4}{2}$ such that \eqref{eq:net-improve-2} holds.
	Clearly, $Y$ is also a $\tilde{\delta}$-net of $((\Gamma_{mjnk}^{0}(Z)),\|\bullet\|_{X_{s}})$. 
	
	We now choose $\beta>\frac{1}{4}\kappa_{1}$ sufficiently large, depending on
	$d,\mu,\theta,\kappa$, such that 
	\[
	|Z|\ge\exp\bigg[2^{-(d+1)}\bigg(\frac{\mu\beta}{\theta}\bigg)^{\frac{d}{\alpha}}\bigg]>\exp\bigg[\eta(1+\theta^{-\frac{1}{2\alpha}})^{2d}\bigg]\ge|Y|.
	\]
	Therefore, we can choose two different $q_{1},q_{2}\in Z\subset\mathcal{N}_{\alpha\beta}^{\theta}(B_{r_{0}})\subset L^{\infty}(B_{1})$
	such that there exists $y\in Y$ such that 
	\begin{align*}
		\|\Gamma_{mjnk}(q_{1};0)-y_{mjnk}\|_{X_{s}} & \le\tilde{\delta},\\
		\|\Gamma_{mjnk}(q_{2};0)-y_{mjnk}\|_{X_{s}} & \le\tilde{\delta}.
	\end{align*}
	It follows from Proposition~\ref{prop:matrix-repn} that 
	\[
	\|\Lambda_{q_{1}}-\Lambda_{q_{2}}\|_{\frac{d+4}{2}\rightarrow-\frac{d+4}{2}}\le8\sqrt{2}\tilde{\delta}.
	\]
	Again, we estimate $\tilde{\delta}$ by $\theta$. 
	\begin{itemize}
		\item \textbf{Case 1.} If $\frac{2\kappa^{2}}{\tilde{\delta}}\le\log(1+\frac{1}{\tilde{\delta}})$,
		then \eqref{eq:equation-theta-delta-improve} yields
		\[
		\theta^{-\frac{1}{2\alpha}}\le3\log\bigg(1+\frac{1}{\tilde{\delta}}\bigg)\le3\log\bigg(\frac{2}{\tilde{\delta}}\bigg),
				\]
		namely, 
		\[
		\tilde{\delta}\le2\exp\bigg(-\frac{1}{3}\theta^{-\frac{1}{2\alpha}}\bigg).
		\]
		\item \textbf{Case 2.} If $\frac{2\kappa^{2}}{\tilde{\delta}}\ge\log(1+\frac{1}{\tilde{\delta}})$,
		then \eqref{eq:equation-theta-delta-improve} implies 
		\[
		\theta^{-\frac{1}{2\alpha}}\le\frac{3\kappa^{2}}{\tilde{\delta}},\quad\text{i.e.},\quad\tilde{\delta}\le3\kappa^{2}\theta^{\frac{1}{2\alpha}}.
		\]
	\end{itemize}
	Putting these two cases together, we obtain 
	\[
	\tilde{\delta}\le2\exp\bigg(-\frac{1}{3}\theta^{-\frac{1}{2\alpha}}\bigg)+3\kappa^{2}\theta^{\frac{1}{2\alpha}},
	\]
	which gives \eqref{eq:DN-est-improve}. By the fact that $Z$ is $\theta$-discrete
	subset with respect to the norm $\|\bullet\|_{L^{\infty}(B_{1})}$, we then obtain \eqref{eq:discrete-q1-q2}. The proof is now completed. 
\end{proof}

\section*{Acknowledgements}
Kow and Wang were partly supported by MOST 108-2115-M-002-002-MY3 and 109-2115-M-002-001-MY3. Uhlmann was partly supported by
NSF, a Walker Family Endowed Professorship at UW and a Si-Yuan Professorship at IAS, HKUST.

\end{document}